\documentclass[a4paper,12pt]{article}
\usepackage{latexsym}
\usepackage{amsmath}
\usepackage{amssymb}
\usepackage{enumerate}
\usepackage{mathrsfs}

\makeatletter
 \long\def\@makefntext#1{\parindent .8em\noindent
  \hbox to 1.1em{\hss $^{\@thefnmark}$}#1}
\makeatother

\usepackage{theorem}
\newtheorem{theorem}{Theorem}[section]
\newtheorem{proposition}[theorem]{Proposition}
\newtheorem{lemma}[theorem]{Lemma}
\newtheorem{corollary}[theorem]{Corollary}

\theorembodyfont{\rmfamily}
\newtheorem{proof}{\textmd{\textit{Proof.}}}

\newtheorem{remark}[theorem]{Remark}

\newtheorem{definition}[theorem]{Definition}

\newtheorem{acknowledgement}{\textmd{\textit{Acknowledgements.}}}

\makeatletter

\@addtoreset{equation}{section}
\makeatother

\newcommand{\qedd}{\hfill \Box}

\renewcommand{\det}{\ensuremath{\mathrm{det}}}
\newcommand{\R}{\ensuremath{\mathbb{R}}}

\newcommand{\Ord}{\ensuremath{\mathcal{O}}}
\newcommand{\LL}{\ensuremath{\mathcal{L}}}
\newcommand{\NN}{\ensuremath{\mathcal{N}}}
\newcommand{\PP}{\ensuremath{\mathcal{P}}}
\newcommand{\G}{\ensuremath{\mathcal{G}}}

\def\div{\mathop{\mathrm{div}}\nolimits}
\def\Hess{\mathop{\mathrm{Hess}}\nolimits}
\def\tr{\mathop{\mathrm{tr}}\nolimits}
\def\Sym{\mathop{\mathrm{S}}\nolimits}

\def\id{\mathop{\mathrm{id}}\nolimits}
\newcommand{\lr}[2]{\left\langle{#1},{#2}\right\rangle}

\title{Behaviors of $\varphi$-exponential distributions in Wasserstein geometry and an evolution equation}
\author{Asuka Takatsu
  \thanks{Graduate School of Mathematics, Nagoya University, Nagoya 464-8602, Japan \& Institut des Hautes \'Etudes Scientifiques, Bures-sur-Yvette 91440, France.
({\sf takatsu@math.nagoya-u.ac.jp} )} \\}
\date{\empty}
\begin{document}
\maketitle
\begin{abstract}
A $\varphi$-exponential distribution is a generalization of an exponential distribution associated to functions $\varphi$ in an appropriate class, 
and the space of $\varphi$-exponential distributions has a dually flat structure. 
We study features of the space of $\varphi$-exponential distributions, such as the convexity in Wasserstein geometry and the stability under an evolution equation.
From this study, we provide the new characterizations to the space of Gaussian measures and the space of $q$-Gaussian measures.
\footnote[0]{
{\bf keywords:} $\varphi$-exponential distribution, $q$-Gaussian measure, Wasserstein geometry, evolution equation.} 
\end{abstract}
\section{Introduction}
A Gaussian measure is an exponential distribution on $\R^d$ with mean and covariance matrix parameters.
The space of Gaussian measures  has good behaviors such as   
the flatness in a dual structure which is a Riemannian metric with a pair of connections being orthogonal to each other, 
the convexity in Wasserstein geometry which is a metric geometry on the space of probability measures, 
and the stability under the linear evolution equation called the Fokker--Planck equation.
As for the validity of these behaviors, for instance, 
the dually flat structure on the space of Gaussian measures leads the Cram\'er--Rao lower bound for any unbiased estimator.

Recently,  Naudts~\cite{N} introduced the space of probability measures, which is a generalization of the space of Gaussian measures equipped with a dually flat structure, and generalized the Cram\'er--Rao lower bound,
where each probability measure is determined by the inverse function of the function $t \mapsto \int_1^{t} 1/{\varphi(s)} ds$ on $(0,\infty)$ for an increasing, positive, continuous function $\varphi$ on $(0,\infty)$.
The probability measure is called a {\it $\varphi$-exponential distribution} and coincides with a usual exponential distribution when $\varphi(s)=s$.
The $\varphi$-exponential distribution on $\R^d$ has a finite second moment  if $\varphi$ has a certain limiting  behavior at $0$ and $\infty$,   
which corresponds to the instances of functions $\varphi$ given by $\varphi(s)=s^q$ with $q\in (0,(d+4)/(d+2))$.
When $\varphi(s)=s^q$ except for $q =1$, 
the $\varphi$-exponential distribution with mean and covariance matrix parameters is called a {\it $q$-Gaussian measure}, 
which is a power-law distribution.

The space of $q$-Gaussian measures behaves well due to the convenience of  power-law distributions.
Indeed, the space of $q$-Gaussian measures is  convex in Wasserstein geometry and  stable under the nonlinear evolution equation of porous medium type
(see~\cite[Proposition~5]{OW} and~\cite[Theorem~A]{qT}).
We prove that the convexity and the stability are features only of the space of Gaussian measures and the space of $q$-Gaussian measures. 
In other words,  the condition $\varphi(s)=s^q$ with $q\in (0,(d+4)/(d+2))$ is a necessary and sufficient condition for the space of $\varphi$-exponential distributions on $\R^d$ with mean and covariance matrix parameters to have each of the two  features.

The paper is organized as follows. 
In the next section, we give definitions and properties about  Wasserstein geometry and the $\varphi$-exponential function.
Section~3 concerns a condition for $\varphi$-exponential distributions to have a finite second moment
(Proposition~\ref{kesu}).
The two spaces of $\varphi$-exponential distributions with mean and covariance matrix parameters are introduced in Section~4, 
one of which has a dually flat structure and the other is convex in Wasserstein geometry (Proposition~\ref{isom}).
We discuss when the two spaces coincide with each other (Theorem~\ref{icchi}).
In Section~5, we consider the stability of the two spaces under a certain evolution equation (Theorem~\ref{th:stable} and Corollary~\ref{cor} ).
\begin{acknowledgement}
The author would like to thank Shin-ichi Ohta for many valuable advice and discussions.
She is also grateful to Sumio Yamada for helpful comments.
She is partially supported by JSPS-IH\'ES (EPDI)  fellowship.
\end{acknowledgement}

\section{Preliminaries}\label{sc:pre}

\subsection{Wasserstein geometry}
Let us briefly recall  properties  of ($L^2$-)Wasserstein geometry over $\R^d$.
See \cite{Vi1,Vi2} and references therein for further information.

In this paper, any measure is always supposed to be a Borel measure.
A measure $\mu$ on $\R^d$ is said to be {\it absolutely continuous} with respect to the Lebesgue measure $\LL^d$ on $\R^d$ if there exists a measurable function $\rho$ on $\R^d$ such that $\mu= \rho \LL^d$.
For a measurable map $F$ on $\R^d$ and a  measure $\mu$ on $\R^d$, 
the {\it push-forward measure} $F_{\sharp}\mu$ of $\mu$ thorough $F$ is a measure on $\R^d$ defined by  $F_{\sharp}\mu[B]:=\mu[F^{-1}(B)]$ for all Borel sets $B \subset \R^d$.
Given two measures $\mu=\rho \LL^d, \nu=\sigma\LL^d$ and a differentiable map $F$ on $\R^d$, 
$\nu=F_\sharp \mu$ is equivalent to $\rho=\sigma(F) \det(d F)$ $\mu$-almost everywhere,  where $dF$ is the total differential of $F$.
For any probability measures $\mu$ and $\nu$ on $\R^d$, a {\it coupling} $\pi$ of $\mu$ and $\nu$ is 
a  probability measure on $\R^d \times \R^d$ with marginals $\mu$ and $\nu$.
In other words, we have $\pi[B \times \R^d]=\mu[B]$ and $\pi[\R^d \times B]=\nu[B]$ for all Borel sets $B \subset \R^d$.
We denote by $\PP_2$ the set of probability measures $\mu$ on $\R^d$ with finite second moments, namely 
\[
 \int_{\R^d} |x|^2 d\mu (x) < \infty.
\]

\begin{definition}\label{df:Wass}
We define the {\it Wasserstein distance}  between $\mu$ and $\nu$ in $\PP_2$ by
\begin{equation}\label{defeq:wass}
W_2(\mu,\nu):=\inf\left\{ \left( \int _{\R^d\times \R^d} |x-y|^2 d \pi (x,y) \right)^{\frac12} \Biggm|  \text{$\pi$ : a coupling of $\mu$ and $\nu$}  \right\}. 
\end{equation}
\end{definition}
The function $W_2$ is indeed a distance function on $\PP_2$.
The metric space $(\PP_2,W_2)$ is called the {\it Wasserstein space}, where  any two points are joined by a geodesic. 
Here a {\it geodesic} is a distance-minimizing curve and has a constant speed. 
It is known that any Wasserstein geodesic does not branch, that is,  if two geodesics have a common interval, then they are subintervals of one geodesic.
Moreover, a coupling  achieving the infimum of~\eqref{defeq:wass} always exists, which is called an {\it optimal coupling}.
For these facts, see~\cite[Chapters~6,7]{Vi2}.
We refer to the results of Knott--Smith~\cite{KS} and Brenier~\cite{Bre} about optimal couplings.
We denote by $\id$ the identity map on $\R^d$.
\begin{theorem} {\rm (\cite[Theorems~2.12, 2.16]{Vi1})}\label{thm:bre}
Let $\mu,\nu \in\PP_2$ and $\mu$ be absolutely continuous with respect to the Lebesgue measure. 
If a proper lower semi-continuous convex function $\phi$ on $\R^d$ 
satisfies $[\nabla \phi]_{\sharp} \mu=\nu$,
then $[\id \times  \nabla \phi]_{\sharp} \mu$ is a unique optimal coupling of $\mu$ and $\nu$, and 
$\{[(1-t)\id +t \nabla \phi]_{\sharp} \mu\}_{t\in[0,1]}$ is a unique Wasserstein geodesic from $\mu$ to $\nu$.  
\end{theorem}

\subsection{$\varphi$-exponential function}
We summarize definitions and properties of the $\varphi$-exponential function.
For further details, we refer to~\cite{N} and references therein.

For an increasing, positive, continuous function $\varphi$ on $(0,\infty)$, 
the {\it $\varphi$-logarithmic function} is defined by 
\[
 \ln_{\varphi}(t):=\int_1^t\frac1{\varphi(s)}ds, 
\]
which is increasing, concave and $C^1$ on $(0,\infty)$. 
The constants $l_{\varphi}$ and $L_{\varphi}$ are respectively defined as the infimum and the supremum of $\ln_ {\varphi}$, that is, 
\begin{equation*}
 l_{\varphi} :=\inf_{t>0} \ln_{\varphi}(t) =\lim_{t\downarrow 0} \ln_{\varphi}(t) \in [-\infty,0),\qquad
 L_{\varphi} :=\sup_{t>0} \ln_{\varphi}(t) = \lim_{t\uparrow \infty} \ln_{\varphi}(t) \in(0,+\infty].
\end{equation*}
The function $\ln_{\varphi}$ has the inverse function, which is called the {\it $\varphi$-exponential function} and is defined on $(l_{\varphi},L_\varphi)$. 
This inverse function can be extended to all of $\R$ as 
\begin{equation*}
\exp_{\varphi}(\tau):=
\begin{cases}
0      
&\text{for\ }\tau  \leq l_{\varphi},\\
\ln_{\varphi}^{-1}(\tau)     
&  \text{for\ } \tau \in \left(l_{\varphi},L_{\varphi} \right),\\
\infty 
&  \text{for\ } \tau \geq L_\varphi,  
\end{cases}
\end{equation*}
which is $C^1$ on $(l_{\varphi},L_{\varphi})$ and its derivative is given by 
\[
  \frac{d}{d\tau} \exp_{\varphi}(\tau)=\varphi(\exp_{\varphi} (\tau)).
\]
We mention that if $\varphi(0)=0$ then the function $\exp_{\varphi}$ is always $C^1$ on $(-\infty,L_\varphi)$.
 
The case of $\varphi(s)=s$ is the most fundamental case related to  the Boltzmann--Gibbs statistics and the Fokker--Planck equation, 
where the $\varphi$-exponential function coincides with the usual exponential function. 
Another important case is that $\varphi(s)=s^{q}$, where the $\varphi$-exponential function is the power function given by  
\[
  \exp_q (\tau):= 
 [ 1+(1-q) \tau]_+^{\frac{1}{1-q}}, 
\]
where we set $[\tau]_+:=\max\{\tau,0\}$ for $\tau\in\R$ and  $0^{a}:=\infty$ for $a<0$. 
This case is related to the Tsallis statistics and  the nonlinear evolution equation of porous medium type.  

We introduce a class of increasing, positive, continuous functions on $(0,\infty)$.
\begin{definition}
For any $a\in\R$, we define $\Ord(a)$ to be the set of all  increasing, positive, continuous functions $\varphi$ on $(0,\infty)$ such that  
$\max\{\delta_\varphi, \delta^\varphi\}<a$, where we set
\begin{gather*}
\delta_\varphi
:=\inf 
    \left\{\delta \in \R \Bigm| \lim_{s \downarrow 0}
               \frac{s^{1+\delta}}{\varphi(s)}  \text{\ exists} 
    \right\},\qquad 
\delta^\varphi
:=\inf 
    \left\{\delta \in \R \Bigm|
              \lim_{s \uparrow \infty} \frac{s^{1+\delta}}{\varphi(s)} =\infty
   \right\}.
\end{gather*}
\end{definition}
It is trivial that $\Ord(a) \subset \Ord(b)$ if $a<b$.
We define the constant $\delta'_\varphi$, expressing the order of $\ln_\varphi$ at $0$, by 
\[
\delta'_{\varphi}
:= \inf 
    \left\{  \delta \in \R  \Bigm|
                \lim_{t \downarrow 0} t^{\delta} \ln_{\varphi} (t) \text{\ exists}
    \right\}. 
\]
We refer to the relation between $\delta_\varphi$ and $\delta'_\varphi$.
\begin{lemma}\label{lem:case}
For any $\varphi \in \Ord(a)$ with some $a\in \R$, we have  $\delta'_{\varphi} \leq \max\{\delta_\varphi, 0\}$.
\end{lemma}
\begin{proof}
Given any  $\varepsilon>0$, there exists $C>0$ such that 
\[
   0 \leq \frac{s^{1+\max\{\delta_\varphi, 0\}+\varepsilon }}{\varphi(s)} \leq C
\] 
holds for any $s \in (0,1)$. 
Set $\delta :=\max\{\delta_\varphi, 0\} + 2 \varepsilon$. 
Then for $t \in(0,1)$, we have 
\begin{align*}
0 
\leq  - t^{\delta} \ln_{\varphi} (t)
= t^{\delta} \int_t^1 \frac1 {\varphi(s)}ds 
\leq \int_t ^1 \frac{s^{\delta}}{\varphi(s)}ds 
\leq  C\int_t ^1 s^{-1+\varepsilon}ds  
= \frac{C}{\varepsilon} ( 1-t^{\varepsilon})  
<  \frac{C}{\varepsilon} < +\infty,  
\end{align*}
proving  $\delta'_\varphi \leq \max\{\delta_\varphi, 0\} + 2 \varepsilon$. 
Since $\varepsilon >0$ is arbitrary, we have $\delta'_\varphi \leq \max\{\delta_\varphi, 0\}$.
$\qedd$
\end{proof}

\section{Condition of $\varphi$}
This section is devoted to the study of  $\varphi$-exponential distributions with mean and covariance matrix parameters.
We first observe behaviors of the following function $f_\varphi$.
\begin{lemma}\label{lem:mass} 
For  any $\varphi \in \Ord(1)$, 
we set the function and the constant as  
\begin{align*}
& f_\varphi(p,\lambda)
:=\int_0^{\exp_\varphi(\lambda)} \left| (\lambda-\ln_\varphi (t))^{p} \frac{t}{\varphi(t)} \right|dt 
=\int_0^{\exp_\varphi(\lambda)} (\lambda-\ln_\varphi (t))^{p} \frac{t}{\varphi(t)}dt, \\
& p_\varphi:=\left\{
\begin{array}{cl}
\displaystyle \frac1{\max\{\delta_\varphi,\delta^{\varphi}\}}-1
&\text{if } \max\{\delta_\varphi,\delta^{\varphi}\} >0, \\
\infty
&\text{otherwise.} 
\end{array}\right.
\end{align*}
\begin{enumerate}[{\rm (1)}]
\item\label{i}
For any $(p,\lambda) \in (-1,p_\varphi) \times (l_\varphi,L_\varphi)$,  $f_\varphi(p,\lambda)$ is well-defined.
\item\label{ii}
The function $\lambda \mapsto f_\varphi(p,\lambda)$ on $(l_\varphi,L_\varphi)$ is $C^0$  for $p\in (-1, 0)$ and  $C^1$ for $p\in [0,p_\varphi )$.
\item\label{iii}
The function 
$p \mapsto f_\varphi(p,\lambda)$ on $(-1,p_\varphi )$ is $C^\infty$  for  $\lambda \in (l_\varphi,L_\varphi)$.
\item\label{iiii}
For any $p \in (-1,p_\varphi )$, $\lim_{\lambda \downarrow l_\varphi } f_\varphi(p,\lambda)=0$ and $\lim_{\lambda \uparrow L_\varphi} f_\varphi(p,\lambda)=\infty$ hold.
\end{enumerate}
\end{lemma}
\begin{proof}
Note that for any $p\in (-1,p_\varphi)$, we have $1/(p+1) > \max\{\delta_\varphi, \delta^{\varphi}, \delta'_\varphi\}$. 

\eqref{i}: Fix any $\lambda \in (l_\varphi,L_\varphi)$ and  take $\varepsilon \in (0, \lambda-l_\varphi)$.
Given any $p\in(-1,0)$, we deform 
\begin{align*}
 f_\varphi(p,\lambda)
&=\int_0^{\exp_\varphi(\lambda-\varepsilon)} (\lambda-\ln_\varphi (t))^{p} \frac{t}{\varphi(t)} dt 
  +\int_{\exp_\varphi(\lambda-\varepsilon)}^{\exp_\varphi(\lambda)}  (\lambda-\ln_\varphi (t))^{p} \frac{t}{\varphi(t)}  dt\\
&=\int_0^{\exp_\varphi(\lambda-\varepsilon)} \left(\lambda-\ln_\varphi (t)\right)^{p} \frac{t}{\varphi(t)}dt
+\int_0^{\varepsilon} s^{p} \exp_\varphi(\lambda-s)ds \\
&\leq  \varepsilon^p \int_0^{\exp_\varphi(\lambda-\varepsilon)}  \frac{t}{\varphi(t)}dt +\exp_\varphi(\lambda)\int_0^{\varepsilon} s^{p} ds,
\end{align*}
where we use the change of variables formula for $s=\lambda-\ln_\varphi(t)$ in the second line, and 
the monotonicity of $\ln_\varphi$ and $\exp_\varphi$ in the inequality.
The integrability of the last line follows from the conditions that $\varphi \in \Ord(1)$ and $p > -1$.

For any $p\in[0,p_\varphi)$,  we fix $\delta$ as 
\[
 \max\{\delta_\varphi, \delta^{\varphi}, \delta'_\varphi\} < \delta:=\frac12\left(  \max\{\delta_\varphi, \delta^{\varphi}, \delta'_\varphi\}+\frac1{p+1} \right) <\frac1{p+1}
\]
and deform $f_\varphi(p,\lambda)$ as 
\[
   f_\varphi(p,\lambda)
 =\int_0^{\exp_\varphi(\lambda)} \left( t^{\delta}  \lambda-t^{\delta} \ln_{\varphi} (t)\right)^{p} 
  \frac{ t^{1+\delta}}{\varphi(t)} 
 \cdot t^{-\delta ( p+1)}dt.   
\]
Since  $\delta (p+1)<1$ holds and the function $ t\mapsto ( t^{\delta}  \lambda-t^{\delta} \ln_{\varphi} (t))^{p}   t^{1+\delta}/\varphi(t)$ is bounded on $(0,\exp_\varphi (\lambda))$, 
$f_\varphi(p,\lambda)$ is well-defined.

\eqref{ii}, \eqref{iii}: The assertions follow from Lebesgue's dominated convergence theorem.

\eqref{iiii}: The assertion $\lim_{\lambda \downarrow l_\varphi } f_\varphi(p,\lambda)=0$ also follows from Lebesgue's dominated convergence theorem.
Integrating by part with the condition $1/(p+1)>\delta'_\varphi$ yields  that  
\begin{align*}
f_\varphi(p,\lambda)
&=\frac{-t}{p+1} \left(\lambda-\ln_\varphi(t)\right)^{p+1} \bigg|_0^{\exp_\varphi(\lambda)} 
+\frac{1}{p+1}\int_0^{\exp_\varphi(\lambda)}(\lambda-\ln_\varphi(t))^{p+1} dt \\
&\geq 0+\frac{1}{p+1}\int_0^{\exp_\varphi(\lambda)} \left( \frac{\exp_\varphi(\lambda)-t}{\varphi(\exp_\varphi(\lambda))} \right)^{p+1} dt \\
&=\frac{\exp_\varphi(\lambda)}{(p+1)(p+2)} \left( \frac{\exp_\varphi(\lambda)}{\varphi(\exp_\varphi(\lambda))} \right)^{p+1}, 
\end{align*}
where the inequality follows from the concavity of $\ln_\varphi$, that is,     
\[
\ln_\varphi(t) 
\leq \ln_\varphi(\exp_\varphi(\lambda))+ \frac{t-\exp_\varphi(\lambda)}{\varphi(\exp_\varphi(\lambda))}   
=\lambda+ \frac{t-\exp_\varphi(\lambda)}{\varphi(\exp_\varphi(\lambda))}.
\]
The condition $1/(p+1) > \delta^\varphi$ leads that
\[
 \lim_{\lambda \uparrow L_\varphi} \exp_\varphi(\lambda) \left( \frac{\exp_\varphi(\lambda)}{\varphi(\exp_\varphi(\lambda))} \right)^{p+1} 
=\lim_{s \uparrow \infty}  \left( \frac{s^{1+\frac1{p+1}}}{\varphi(s)} \right)^{p+1}
=\infty, 
\]
proving $\lim_{\lambda \uparrow L_\varphi} f_\varphi(p,\lambda)=\infty$. 
$\qedd$
\end{proof}

Let $\Sym(d,\R)_+$ be the set of symmetric positive definite matrices of size $d$, and  $\{e_i\}_{i=1}^d$ be the standard basis on $\R^d$.
Given any $V \in \Sym(d,\R)_+$,  
let $V^{1/2}$ be the symmetric positive definite matrix  such that  $V^{1/2}\cdot  V^{1/2} = V$ and  $|x|_V^2:=\lr{x}{V^{-1}x}$ for $x\in\R^d$.
\begin{proposition}\label{kesu}
For any $\varphi \in \Ord(2/(d+2))$ with $d\geq2$, 
there exist constants $\lambda \in (l_\varphi,L_\varphi)$ and $c>0$ such that 
\begin{align}
\label{mass}
&\int_{\R^d} \exp_\varphi (\lambda-c|x-v|_V^2) d\LL^d =1, \\
\label{mean}
&\int_{\R^d} \lr{x}{e_i} \exp_\varphi (\lambda-c|x-v|_V^2) d\LL^d= \lr{v}{e_j}, \\
\label{vari}
&\int_{\R^d} \lr{x-v}{e_i}\lr{x-v}{e_j} \exp_\varphi (\lambda-c|x-v|_V^2) d\LL^d= \lr{e_i}{Ve_j}.
\end{align}
\end{proposition}
\begin{proof}
Note that the condition $\varphi \in \Ord(2/(d+2))$ with $d \geq 2$ implies that $1 \leq d/2 < p_\varphi $.
Let $\Gamma(\cdot)$ be the Gamma function.
Applying the change of variables formula first for $y=c^{1/2}V^{-1/2}(x-v)$ then from Cartesian to polar coordinates 
and for $t=\exp_{\varphi}(\lambda-|y|^2)$, we have   
\begin{align*}
&\int_{\R^d} \exp_\varphi (\lambda-c|x-v|_V^2) d\LL^d 
=f_\varphi\left( \frac{d-2}{2}, \lambda\right) \cdot\left( \frac{\pi}{c} \right)^{\frac{d}2}  \frac{\sqrt{ \det V }}{\Gamma\left( d/2 \right)}, \\
&\int_{\R^d} \lr{x}{e_i} \exp_\varphi (\lambda-c|x-v|_V^2) d\LL^d 
=f_\varphi\left( \frac{d-2}{2}, \lambda\right) \cdot\left( \frac{\pi}{c} \right)^{\frac{d}2}  \frac{\sqrt{ \det V }}{\Gamma\left( d/2 \right)}\cdot\lr{v}{e_i}, \\
&\int_{\R^d} \lr{x-v}{e_i}\lr{x-v}{e_j} \exp_\varphi (\lambda-c|x-v|_V^2) d\LL^d 
=f_\varphi\left(\frac{d}2, \lambda\right) \cdot \left( \frac{\pi}{c} \right)^{\frac{d}2}  \frac{\sqrt{ \det V }}{ \Gamma\left( d/2 \right)} \cdot \frac{\lr{e_i}{Ve_j}}{cd},
\end{align*}
which are well-defined by Lemma~\ref{lem:mass}. 
This means that the pair $(\lambda,c)$ satisfying 
\begin{gather}\label{const}
\frac{f_\varphi\left((d-2)/2,\lambda\right)^{\frac{d+2}2}} { f_\varphi\left(d/2,\lambda\right)^{\frac{d}2} }
=(d\pi)^{-\frac{d}2}  \frac{\Gamma\left(d/2 \right)} {\sqrt{\det  V}}, \qquad 
c=\frac{f_\varphi\left(d/2,\lambda\right)}{df_\varphi\left((d-2)/2,\lambda\right)}
\end{gather}
is the desired one.
If the continuous function $F_\varphi(p,\lambda):=f_\varphi(p-1,\lambda) ^{p+1}/f_\varphi(p,\lambda)^p$ on $(0,p_\varphi) \times (l_\varphi,L_\varphi)$ 
satisfies 
\begin{equation}\label{limit}
\lim_{\lambda \downarrow l_\varphi } F_\varphi\left(\frac{d}2,\lambda\right)=0,\qquad
\lim_{\lambda \uparrow L_\varphi} F_\varphi\left(\frac{d}2,\lambda\right)=\infty,
\end{equation}
then the intermediate value theorem guarantees the existence of such a pair $(\lambda,c)$.

The rest is to prove~\eqref{limit}.
For any $(p,\lambda) \in (-1, p_\varphi) \times (l_\varphi,L_\varphi)$, H\"older's inequality yields that
\begin{align*}
\frac{\partial}{\partial p}f_\varphi (p,\lambda) 
&=\int_0^{\exp_\varphi(\lambda)} \ln (\lambda-\ln_\varphi(t)) \frac{ (\lambda-\ln_\varphi(t))^p t}{\varphi(t)} dt  \\
&\leq \left( \int_0^{\exp_\varphi(\lambda)} \ln^2 (\lambda-\ln_\varphi(t)) \frac{(\lambda-\ln_\varphi(t))^p t}{\varphi(t)} dt   \right)^{\frac12}
     \left( \int_0^{\exp_\varphi(\lambda)}  \frac{(\lambda-\ln_\varphi(t))^p t}{\varphi(t)} dt   \right)^{\frac12} \\
&=\left(\frac{\partial^2}{\partial p^2}f_\varphi (p,\lambda) \right)^{\frac12} f_\varphi(p,\lambda)^{\frac12}, 
\end{align*}
providing the convexity of $\ln f_\varphi$ in $p$.
Then for any $(p,\lambda) \in (0, p_\varphi) \times (l_\varphi,L_\varphi)$, we have 
\begin{align*}
\frac{\partial}{\partial p}\ln F_\varphi(p,\lambda)
&= (p+1) \frac{\partial}{\partial p}\ln f_\varphi(p-1,\lambda)+ \ln f_\varphi(p-1,\lambda) - p \frac{\partial}{\partial p} \ln f_\varphi(p,\lambda) - \ln f_\varphi(p,\lambda) \\
&\leq  p \frac{\partial}{\partial p}\ln f_\varphi(p-1,\lambda) - p \frac{\partial}{\partial p} \ln f_\varphi(p,\lambda) 
\leq 0 
\end{align*}
and deduce the monotonicity of  $F_\varphi$ in $p$.
Since we have
\[
\frac{\partial}{\partial \lambda}f_\varphi(0,\lambda) 
= \exp_\varphi(\lambda), \qquad 
\frac{\partial}{\partial \lambda}f_\varphi(p,\lambda) 
=p\int_0^{\exp_\varphi(\lambda)} (\lambda-\ln_\varphi(t))^{p-1} \frac{t}{\varphi(t)} dt
=pf_\varphi(p-1,\lambda) 
\]
for $p\in(1,p_\varphi)$,  
de l'H\^opital's rule with Lemma~\ref{lem:mass}\eqref{iiii} yields that
\begin{align*}
&\lim_{\lambda \downarrow l_\varphi } F_\varphi\left(1,\lambda\right) = \lim_{\lambda \downarrow l_\varphi } 2\exp_\varphi(\lambda)=0, 
&&\lim_{\lambda \uparrow L_\varphi } F_\varphi\left(1,\lambda\right) = \lim_{\lambda \uparrow L_\varphi } 2\exp_\varphi(\lambda)=\infty,\\
&\lim_{\lambda \downarrow l_\varphi } F_\varphi\left(p,\lambda\right)^{\frac1p}=\frac{p^2-1}{p^2}\lim_{\lambda \downarrow l_\varphi } F_\varphi(p-1,\lambda)^{\frac{1}p}, 
&&\lim_{\lambda \uparrow L_\varphi } F_\varphi\left(p,\lambda\right)^{\frac1p}=\frac{p^2-1}{p^2}\lim_{\lambda \uparrow L_\varphi } F_\varphi(p-1,\lambda)^{\frac{1}p} 
\end{align*}
for any $p\in(1,p_\varphi)$, in particular, 
\begin{align*}
\lim_{\lambda \downarrow l_\varphi } F_\varphi\left(n,\lambda\right) = 
\lim_{\lambda \downarrow l_\varphi } F_\varphi\left(1,\lambda\right) =0, \qquad
\lim_{\lambda \uparrow L_\varphi } F_\varphi\left(n,\lambda\right) =
\lim_{\lambda \uparrow L_\varphi } F_\varphi\left(1,\lambda\right) = \infty
\end{align*}
for any integer $n \in[1, p_\varphi)$.
This with the monotonicity of $F_\varphi$ in $p$ finishes the proof. 
$\qedd$
\end{proof}

\begin{remark}\label{rem}
For the pair of constants $(\lambda,c)$  obtained for $V=I_d$ in~\eqref{const}, 
we find that $\exp_{\varphi}(\lambda-c |x-v|_V^2) (\det V)^{-1/2} \LL^d$ becomes the probability measure with mean $v$ and covariance matrix $V$ 
by the change of variables formula for $y=V^{-1/2}(x-v)$. 
\end{remark}

\section{Convexity in Wasserstein geometry}
In this section, we investigate the two spaces of $\varphi$-exponential distributions with mean and covariance matrix  parameters.
We mention that $\alpha \varphi(s)=s^{q} \in \Ord(2/(d+2))$ with some $\alpha>0$ if and only if $q \in Q_d:=(0,(d+4)/(d+2))$.

Proposition~\ref{kesu} and Remark~\ref{rem} yield that for any $\varphi \in \Ord(2/(d+2))$ with $d \geq 2$,
there exist continuous functions $\lambda_\varphi$ and $c_\varphi$ on $\Sym(d,\R)_+$ such that each of 
\begin{align*}
n_{\varphi}(v,V)(x)&:=\exp_{\varphi}\left(\lambda_\varphi(V)-c_\varphi(V) |x-v|_V^2\right), \\
g_{\varphi}(v,V)(x)&:=\exp_{\varphi}(\lambda_\varphi(I_d)-c_\varphi(I_d) |x-v|_V^2) (\det V)^{-\frac12}
\end{align*}
is the probability density on $(\R^d, \LL^d)$ with mean $v$ and covariance matrix $V$.
We define the two spaces of probability measures by
\begin{gather*}
\mathcal{N}_{\varphi}:=\{N_{\varphi}(v,V):=n_{\varphi}(v,V)\LL^d \bigm| (v,V) \in \R^d \times \Sym(d,\R)_+ \}, \\
\mathcal{G}_{\varphi}:=\{G_{\varphi}(v,V):=g_{\varphi}(v,V)\LL^d \bigm| (v,V) \in \R^d \times \Sym(d,\R)_+ \},
\end{gather*}
which generally differ from each other.

\begin{theorem}\label{icchi}
For any $\varphi, \psi \in\Ord(2/(d+2))$ with $d \geq 2$, 
$G_{\varphi} (0,a^{2}I_d)=N_{\psi}(0,a^{2}I_d) $ holds for any $a>0$ if and only if $\alpha\varphi(s)=\beta\psi(s)=s^{q}$ with some $\alpha,\beta >0$ and $q \in Q_d$.
\end{theorem}
\begin{proof}
We first prove the ``if" part.
For any $\varphi \in\Ord(2/(d+2))$ with $d \geq 2$ and  $\alpha>0$,  the following relations are verified;
\begin{gather*}
\ln_{\alpha \varphi}(t)=\alpha^{-1} \ln_{\varphi}(t), \qquad 
\exp_{\alpha \varphi}(\tau)=\exp_{\varphi}(\alpha \tau), \quad 
f_{\alpha \varphi}(p,\lambda)=\alpha^{-(p+1)}f_{\varphi}(p,\alpha\lambda). 
\end{gather*}
These relations with~\eqref{const} implies that, for any $V\in \Sym(d,\R)_+$,  we have 
\begin{gather*}
(d\pi)^{-\frac{d}2}  \frac{\Gamma\left(d/2 \right)} {\sqrt{\det  V}}=F_{\alpha \varphi}\left(\frac{d}2, \lambda_{\alpha \varphi}(V)\right)=F_{\varphi}\left(\frac{d}2, \alpha \lambda_{ \alpha \varphi}(V)\right),\\ 
c_{\alpha\varphi}(V)
=\frac{f_{\alpha\varphi}\left(d/2, \lambda_{\alpha\varphi} (V)\right)}{df_{\alpha\varphi} \left((d-2)/2, \lambda_{\alpha\varphi}(V)\right)}
=\alpha^{-1}\frac{f_{\varphi}\left(d/2, \alpha\lambda_{\alpha\varphi} (V)\right)}{df_{\varphi} \left((d-2)/2, \alpha\lambda_{\alpha\varphi}(V)\right)},
\end{gather*}
which shows  that $\alpha \lambda_{\alpha\varphi} (V) =\lambda_{\varphi}(V)$ and  $\alpha c_{\alpha \varphi}(V) =c_{\varphi}(V)$, 
in turn $G_{\alpha\varphi}(v,V)=G_{\varphi}(v,V)$ and $N_{\alpha\varphi}(v,V)=N_{\varphi}(v,V)$.
Therefore it is enough to prove the case $\varphi(s)=\psi(s)=s^q$ with $q\in Q_d$, which  has been already demonstrated in~\cite[Section~4]{qT}.

Let us  prove the ``only if" part. 
The condition $G_{\varphi} (0,a^{2}I_d)=N_{\psi}(0,a^{2}I_d)$ leads that  
\begin{align*}
 a^{-d}  \exp_{\varphi}(\lambda-ca^{-2}|x|^2) =\exp_{\psi}(\lambda(a)-c(a)a^{-2} |x|^2) 
\end{align*}
for any $x\in\R^d$, where we abbreviate $c:=c_\varphi(I_d), \lambda:=\lambda_\varphi(I_d)$ and $c(a):=c_\psi(a^{2} I_d), \lambda(a):=\lambda_\psi(a^{2} I_d)$.
Evaluating the above equation at $x\in\R^d$ with $|x|=ar$, we deduce that    
\begin{gather*}
a^{-d} \exp_{\varphi}(\lambda-cr^2 ) =   \exp_{\psi}(\lambda(a)-c(a)r^2), \qquad
b :=\sqrt{\frac{\lambda- l_{\varphi}}{c}}=\sqrt{\frac{\lambda(a)-l_{\psi}}{c(a)}}.
\end{gather*}
Differentiating this equation at $r\in I :=(0,b)$ yields that 
\[
 ca^{-d}\varphi ( \exp_{\varphi}(\lambda-cr^2) ) = c(a) \psi ( \exp_{\psi}(\lambda(a)-c(a)r^2 ) )   = c(a)  \psi  (  a^{-d} \exp_{\varphi}(\lambda-cr^2 )),
\] 
in particular,   $c \varphi ( \exp_{\varphi} (\lambda-cr^2 ) ) =c(1)  \psi  ( \exp_{\varphi}(\lambda-cr^2 ))$ for the case $a=1$.
This provides 
\begin{align}\label{yorazu}
\frac{\psi  (  a^{-d} \exp_{\varphi}(\lambda-cr^2 ) )}{\psi  ( \exp_{\varphi}(\lambda-cr^2 ))} 
=\frac{c(1)\psi  (  a^{-d} \exp_{\varphi}(\lambda-cr^2 ) )}{ c \varphi ( \exp_{\varphi} (\lambda-cr^2 ) )} 
=\frac{c(1)}{c(a)a^d} 
\end{align}
for any $r \in I$.
For any $\xi,\eta>0$ satisfying $\xi^2, \xi \eta <\exp_\varphi(\lambda)$, 
since it is possible to choose $s,r \in I$ such that $\xi^2=\exp_{\varphi}(\lambda-cr^2 )$ and $\xi \eta =\exp_{\varphi}(\lambda-c s ^2 )$, 
the choice $a^{-d}=\eta/\xi$ leads 
\[
\frac{\psi  (\xi\eta )}{\psi  ( \xi^2)} 
=\frac{\psi  (  a^{-d} \exp_{\varphi}(\lambda-cr^2 ) )}{\psi  ( \exp_{\varphi}(\lambda-cr^2 ))} 
=\frac{c(1)}{c(a)a^d}
=\frac{\psi  (\eta^2 )}{\psi  ( \xi\eta)}.
\]
Then for any $\alpha \in (0,  \min\{\exp_\varphi(\lambda)^{1/2},\exp_\varphi(\lambda)\})$, the function $\Psi(\xi):={\psi(\xi \alpha^2)}/{\psi(\alpha^2)}$
satisfies  that $\Psi(\xi \eta)=\Psi(\xi)\Psi(\eta)$ for any $\xi, \eta \in (0,\alpha^{-1/2}\exp_\varphi(\lambda)^{1/2})$.
By induction with the condition $\alpha<\exp_\varphi(\lambda)$, we find that $\Psi(\xi \eta)=\Psi(\xi)\Psi(\eta)$ for any $\xi, \eta >0$.
Hence the continuous function $L(s):=\ln \Psi(e^s) $ on $\R$ satisfies  Cauchy's  functional equation, namely $L(s+t)=L(s)+L(t)$.
This means that $L(s)=L(1)s$, in turn $\alpha\varphi(s)=\beta\psi(s)=s^q$ for some $\alpha,\beta>0$ and $q \in Q_d$.
$\qedd$
\end{proof}
\begin{remark}
Theorem~\ref{icchi} yields that for $\varphi, \psi \in\Ord(2/(d+2))$ with $d \geq 2$,  
we have $G_{\varphi} (0,a^2 I_d)=N_{\psi}(0,a^2 I_d)$ for any $a>0$   
if and only if $\G_{\varphi}$ and $\NN_{\psi}$ are either the space of Gaussian measures, or the space of $q$-Gaussian measures.
\end{remark}
%
Both the spaces $\G_{\varphi}$ and $\NN_{\varphi}$ are generalizations of the space of Gaussian measures,  
however they differ from each other in general.
On one hand, the space $\NN_\varphi$ has a dually flat structure as seen in~\cite{N}
(we do not refer to a dually flat structure, see~\cite{Amari}).
On the other hand, the space $\G_\varphi$ is convex in Wasserstein geometry.

\begin{proposition}\label{isom}
For any $\varphi\in\Ord(2/(d+2))$ with $d \geq 2$, the space $\G_{\varphi}$ is convex and isometric to the space $\G$ of Gaussian measures 
in the sense of  Wasserstein geometry.
\end{proposition}
\begin{proof}
For $(v,V), (u,U) \in \R^d\times\Sym(d,\R)_+$, we define the symmetric positive definite matrix $W$ and the associated convex function $\phi_W$ by
\begin{gather*}
 W:=U^{\frac12}(U^{\frac12}V U^{\frac12})^{-\frac12}U^{\frac12}, 
\qquad
\phi_W (x) :=\frac{1}{2} \langle x-v,W(x-v)\rangle +\langle x,u\rangle.
\end{gather*}
Then $\phi_W$ is convex and satisfies $[\nabla \phi_W]_\sharp G_\varphi(v,V)=G_\varphi(u,U)$ 
since for $G_\varphi(v,V)$-almost everywhere we have $g_\varphi(v,V)= g_\varphi(u,U)(\nabla \phi_W)(\det \Hess \phi_W)$.
Theorem~\ref{thm:bre}  yields that $[\id \times \nabla \phi_W]_{\sharp} G_\varphi(v,V)$ is an optimal coupling of $G_\varphi(v,V)$ and $G_\varphi(u,U)$, 
consequently  
\begin{align*}
W_2(G_\varphi(v,V),G_\varphi(u,U))^2 
&=|v-u|^2 +\tr V + \tr U 
 -2 \tr \left( U^{\frac12}V U^{\frac12} \right)^{\frac12} \\
&=W_2(G(v,V),G(u,U))^2,   
\end{align*}
where $G(v,V)$ stands for the Gaussian measure with mean $v$ and covariance matrix $V$.
(Note that  the optimality of the function $\phi_W$ in the case of the Gaussian measures has been known, for instance, see~\cite{T} and references therein.)
Hence the map $G(v,V) \mapsto G_{\varphi}(v,V)$ is an isometry from $\G$ to $\G_{\varphi}$ in Wasserstein geometry.
We also deduce from Theorem~\ref{thm:bre} that, for the time-dependent vector $\{w_t\}_{t\in[0,1]}$ and the time-dependent matrix $\{W_t\}_{t\in[0,1]}$ defined by 
\[
w_t:=(1-t)v+t u,\qquad
W_t:=[(1-t)I_d+t W]V[(1-t)I_d+t W],
\]
$\{G_\varphi(w_t,W_t)\}_{t \in [0,1]}$ is a unique Wasserstein geodesic from $G_\varphi(v,V)$ to $G_\varphi(u,U)$.
Hence the space $\G_{\varphi}$ is convex in Wasserstein geometry.
$\qedd$
\end{proof}

Let us now consider the completions of $\NN_{\varphi}$ as a metric space with respect to $W_2$, denoted by $\overline{\NN_{\varphi}}$, 
and prove that $\overline{\NN_{\varphi}}$ is generally not convex in Wasserstein geometry.
By Proposition~\ref{isom}, the completion of $\G_{\varphi}$, denoted  by $\overline{\G_{\varphi}}$, is isometric to the completion of  $\G$, 
which is homeomorphic to $\R^d \times \Sym(d,\R)_{\geq 0}$, 
where $\Sym(d,\R)_{\geq 0}$ is the set of all symmetric non-negative definite matrices of size $d$ (see~\cite[Section~4]{T}).
We similarly show that $\overline{\NN_{\varphi}}$ is homeomorphic to $\R^d \times \Sym(d,\R)_{\geq 0}$, whose proof is simple  but tedious and we omit it.
The key of the proof is the fact that 
the weak convergence, which is equivalent to the pointwise convergence of  characteristic functions, is weaker than the convergence in Wasserstein geometry (see~\cite[Theorem~7.12]{Vi1}). 
Hence elements of  $\overline{\G_{\varphi}}$ and $\overline{\NN_{\varphi}}$ 
are naturally denoted by $G_\varphi(v,V)$ and $N_{\varphi} (v,V)$ with some $(v,V )\in \R^d \times \Sym(d,\R)_{\geq0}$, respectively. 
Note that any elements in  $\overline{\G_{\varphi}}\setminus \G_{\varphi} $ and $\overline{\NN_{\varphi}} \setminus \NN_{\varphi}$ are singular with respect to the Lebesgue measure  and 
$G_\varphi(v,0)=N_{\varphi} (v,0)=\delta_v$, which is the Dirac measure centered at $v$.

\begin{proposition}\label{th:cvx}
The space  $\overline{\NN_{\varphi}}$ is convex in Wasserstein geometry if and only if $\alpha\varphi(s)= s^q$ with some $\alpha >0$ and $q\in Q_d$.
\end{proposition}
\begin{proof}
Since the ``if" part is trivially true, we assume the convexity of $\overline{\NN_{\varphi}}$.  
Then a unique geodesic $\sigma$ from $N_{\varphi} (0,0)=G_\varphi(0,0)$ to $N_{\varphi}(0,I_d)=G_\varphi(0,I_d)$ is contained in $\overline{\NN_{\varphi}} $ and $\overline{\G_{\varphi}}$, which is written as $\sigma(t)=G_{\varphi}(0, t^2I_d)$ for $t\in[0,1]$ and can be extended to $[0,\infty)$.
The non-branching property in Wasserstein geometry ensures that $G_{\varphi}(0,t^2I_d) \in \G_{\varphi} \cap \NN_{\varphi}$ for any $t>0$ and by Theorem~\ref{icchi}, 
$\alpha\varphi(s)=s^{q}$ holds for some $\alpha>0$ and $q \in Q_d$. 
$\qedd$
\end{proof}

\section{Stability under an evolution equation}
We finally discuss when the space of $\varphi$-exponential distributions is stable  under the evolution equation of the form 
\begin{equation}\label{eq:flow}
\frac{\partial}{\partial t} \rho= \div \big( \rho \nabla (\ln_{\varphi} (\rho) +\Psi_\varphi) \big)  
                                        = \div\left( \frac{\rho \nabla \rho}{ \varphi ( \rho)} + \rho \nabla\Psi_\varphi \right), \qquad
\Psi_\varphi(x)=c_\varphi(I_d) |x|^2.
\end{equation}
We mention that this evolution equation recovers the Fokker--Planck equation 
(resp.\ the nonlinear evolution equation of porous medium type) when $\varphi(s)=s$ 
(resp.\  $\varphi(s)=s^q$  with some $q\in Q_d $ except for $q =1$).

A set of functions is said to be {\it stable} under an evolution equation 
if any solution to the evolution equation with initial data being  an element of the set stays inside the set.
A set of absolutely continuous measures with respect to the Lebesgue measure is said to be {\it stable} under an evolution equation 
if the set generated by densities of elements in the given set is stable.

\begin{theorem}\label{th:stable}
Assume $\varphi(0)=0$.
Then the space $\NN_{\varphi}$ is stable under \eqref{eq:flow} if and only if  $\alpha\varphi(s)=s^q$ with some $\alpha>0$ and $q \in Q_d$.
\end{theorem}
\begin{proof} 
Since the ``if" part has been already shown in~\cite[Proposition~5]{OW},  it is enough to prove the ``only if" part, 
which is done in a similar way to the proof of ~\cite{OW}.

Assume the stability of $\NN_\varphi$.
Then the solution to~\eqref{eq:flow} with initial data $n_{\varphi}(0,V) \in \NN_\varphi$ is written as $n_\varphi(v_t, V_t)$ for 
some time-dependent vector $v_t \in \R^d$ and some time-dependent matrix $V_t \in  \Sym(d,\R)_+$.
The following calculation can be justified; 
\begin{align*}
\frac{d}{dt} v_t
=\int_{\R^d}  x \frac{\partial} {\partial t} n_{\varphi}(v_t,V_t)  d\LL^d 
&=- \int_{\R^d}  n_{\varphi}(v_t,V_t)  \nabla \left( \ln_{\varphi} (n_{\varphi} (v_t,V_t) ) +\Psi_\varphi \right) d\LL^d  
=-2c_\varphi(I_d) v_t,
\end{align*}
which implies that $v_t$ is identically the zero vector.
We similarly compute 
\begin{align}\label{mat}
\frac{d}{dt}V_t=4 A_tV_t, \qquad
\frac{d}{dt} V_t^{-1} =-4 V_t^{-1}A_t, \qquad
A_t:=c(V_t)V_t^{-1}-c_\varphi(I_d)I_d.
\end{align}

In what follows,  we abbreviate $n_\varphi (0,X)$ as $n_\varphi (X)$.
Let $\partial_{ij}$ be the tangent vector  on $\NN_\varphi(0):= \{n_\varphi(X)\ |\ X\in\Sym(d,\R)_+\}$ 
associated to a global chart $(\Sym(d,\R)_+, \Xi=(\Xi_{ij})_{1 \leq i \leq j\leq d} )$ given by $\Xi(X):=n_{\varphi}(X^{-1})$.
We define the map  $\sigma$ on $\Sym(d,\R)_+$ by $\sigma(X)=X^{-1}$  and  
the symmetric matrices  $\Lambda,C$ of size $d$ by 
\[
 \Lambda
:=\begin{cases}
\partial_{ij}  (\lambda_\varphi \circ \sigma) \big|_{V_t^{-1}} & \text{if\ }i=j, \\
\displaystyle \frac12\partial_{ij} (\lambda_\varphi \circ \sigma) \big|_{V_t^{-1}} & \text{if\ }i\neq j,
\end{cases} 
\qquad
C
:=\begin{cases}
 \partial_{ij} (c_\varphi \circ \sigma) \big|_{V_t^{-1}} & \text{if\ }i=j, \\
 \displaystyle \frac12\partial_{ij} (c_\varphi \circ \sigma) \big|_{V_t^{-1}} & \text{if\ }i\neq j.
\end{cases} 
\]
Given any $x\in \R^d$, we set the function $e_x$ on $\NN_\varphi(0)$ as $e_x(n_\varphi (X))=n_{\varphi}(X)(x)$, which is $C^1$ owing to $\varphi(0)=0$.
Then for the solution $n_\varphi(V_t)$ to~\eqref{eq:flow}, we have 
\begin{align*}
  \frac{d}{dt}  e_x( n_\varphi (V_t) ) 
=\frac{\partial}{\partial t} n_\varphi (V_t)(x) 
&=\div\left( n_\varphi (V_t) \nabla ( \ln_\varphi ( n_\varphi (V_t) )+ \Psi_\varphi )\right) (x)  \\
&=4\varphi( n_{\varphi}(V_t) (x)) c(V_t) \lr{x}{V_t^{-1}A_t x} - 2 n_{\varphi}(V_t)(x) \tr (A_t).   
\end{align*}
On the other hand,  by using~\eqref{mat},  a direct computation provides   
\begin{align*}
\frac{d}{dt}  e_x( n_\varphi(V_t)) 
&=\sum_{1 \leq i \leq j\leq d}\partial_{ij} ( e_x \circ \Xi_{ij} ) \big|_{V_t^{-1}} \frac{ d \lr{e_i}{V_t^{-1}e_j} }{dt} \\
&=-4\varphi( n_{\varphi}(V_t)(x) )\left[  \tr  ( V_t^{-1}A_t \Lambda ) - \tr  (V_t^{-1}A_t  C ) |x|_{V_t}^2  - c_\varphi(V_t) \lr{x}{V_t^{-1}A_t x} \right].  
\end{align*}
We therefore obtain 
\begin{align}\label{1226}
 n_{\varphi}(V_t)(x) \tr (A_t)
=2\varphi( n_{\varphi}(V_t) (x)) \left[ \tr  ( V_t^{-1}A_t \Lambda ) - \tr  (V_t^{-1}A_t  C )|x|_{V_t}^2) \right]
\end{align}
for any $x\in \R^d$.
Set 
\[ \beta:=\frac{2\tr ( V_t^{-1}A_t \Lambda )}{ \tr (A_t)}, \qquad 
    \gamma:=\frac{ 2  \tr  (V_t^{-1}A_t  C )}{ c_\varphi(V_t) \tr ( A_t)}, \qquad 
    J_{V_t}:=(0, \lambda_\varphi(V_t)-l_\varphi).
\]    
At $x\in \R^d$ with $c_\varphi(V_t) |x|_{V_t}^2=r \in J_{V_t}$,  the equation~\eqref{1226} is deformed as 
\begin{align*}
 \exp_{\varphi}(\lambda_\varphi(V_t)-r) 
=\varphi( \exp_{\varphi} (\lambda_\varphi(V_t)-r)) \left( \beta + \gamma r  \right) 
=-\left( \frac{d}{dr} \exp_{\varphi}(\lambda_\varphi(V_t)-r) \right) (\beta+\gamma r)  
\end{align*}
and integrating it with respect to $r \in J_{V_t}$ yields     
\[
\exp_{\varphi}(\lambda_\varphi(V_t)-r) 
=\exp_{\varphi}(\lambda_\varphi(V_t)) 
\left( 1+\frac{\gamma r}{\beta}  \right)^{-\frac{1}{\gamma}}.  
\]
Taking $\ln_{\varphi}$  and differentiating it in $r$, we have 
\[
\beta \exp_{\varphi}(\lambda_\varphi(V_t) )^{\gamma} \varphi(s)= s^{1+\gamma},
\]
where we set $s=s(r):=\exp_{\varphi}(\lambda_\varphi(V_t)-r)$.
This implies that $\alpha\varphi(s)= s^{q}$ holds on $(0,\exp_\varphi(\lambda(V_t)))$
with $\alpha=\alpha(V_t):=\beta \exp_{\varphi}(\lambda_\varphi(V_t) )^{\gamma} $ and $q=q(V_t):=1+\gamma$.
We mention that $\alpha$ and $q$ do not depend on $V_t$ 
since we have $\varphi(s)=\alpha(V_t)^{-1} s^{q(V_t)}=\alpha(U_t)^{-1} s^{q(U_t)}$ on $(0, \min\{\exp_\varphi(\lambda(V_t)),\exp_\varphi(\lambda(U_t))\} ) \neq \emptyset$ for any $V,U \in \Sym(d,\R)_+$. 
Letting $V \to 0$ and  $t \to 0$,  we have  $\lambda_\varphi(V_t)  \to L_\varphi$ then 
$\alpha \varphi(s)= s^{q}$ for any $s>0$.
$\qedd$
\end{proof}

The same result holds for $\G_{\varphi}$.
The proof is exactly similar  but tedious and we omit it. 
\begin{corollary}\label{cor}
Let $\varphi$ be $C^1$ and $\varphi(0)=0$. 
Then  the space $\G_{\varphi}$ is stable under \eqref{eq:flow} if and only if  $\alpha \varphi(s)=s^q$ with some $\alpha>0$ and $q \in Q_d$.
\end{corollary}

\end{document}